\documentclass[11pt]{amsart}
\newtheorem{thm}{Theorem}[section]
\newtheorem{cor}[thm]{Corollary}

\newtheorem{lem}[thm]{Lemma}
\newtheorem{prop}[thm]{Proposition}
\theoremstyle{definition}
\newtheorem{defn}[thm]{Definition}
\theoremstyle{remark}
\newtheorem{rem}[thm]{Remark}
\newtheorem{exam}[thm]{Example}
\numberwithin{equation}{section}

\def\ca{\mathcal{ A}}

\def\cc{\mathcal{ C}}

\def\ce{\mathcal{ E}}

\def\ch{\mathcal{ H}}

\def\cj{\mathcal{ J}}
\def\ck{\mathcal{K}}

\def\cm{\mathcal{ M}}

\def\ga{{\frak A}}
\def\gb{{\frak B}}
\def\gc{{\frak C}}

\def\bc{{\mathbb C}}

\def\bn{{\mathbb N}}
\def\bp{{\mathbb P}}

\def\bz{{\mathbb Z}}

\def\a{\alpha}
\def\b{\beta}
  
\def\d{\delta}

\def\l{\lambda} \def\LL{\Lambda}

\def\m{\mu}
\def\p{\psi}

\def\r{\rho}
\def\s{\sigma} 
\def\t{\tau}
\def\f{\varphi} 
\def\v{\phi}
  
\def\w{\omega} \def\O{\Omega}

\def\tr{\mathop{\rm Tr}}
\def\id{{\bf 1}\!\!{\rm I}}

\def\1b{{\mathbf{1}}}

\def\L{\Lambda}

\def\o{\otimes}

\def\a{\alpha}

\begin{document}

\title[Open Quantum Random Walks \& Quantum Markov Chains]{Open Quantum Random Walks, Quantum Markov Chains and Recurrence}

\author{Ameur Dhahri}
\address{Ameur Dhahri\\
 Department of Mathematics, Chungbuk National University\\
Chungdae-ro, Seowon-gu, Cheongju, Chungbuk 362-763, Korea} \email{{\tt ameur@chungbuk.ac.kr}}

\author{Farrukh Mukhamedov}
\address{Farrukh Mukhamedov\\
Department of Mathematical Sciences, College of Science, The United
Arab Emirates University, P.O. Box 15551, Al Ain, Abu Dhabi, UAE}
\email{{\tt farrukh.m@uaeu.ac.ae}}

\begin{abstract}
In the present paper, we construct QMCs associated with Open Quantum
Random Walks such that the transition operator of the chain is
defined by OQRW and the restriction of QMC to the commutative
subalgebra coincides with the distribution $\bp_\r$ of OQRW. This
sheds new light on some properties of the measure $\bp_\r$. As an
example, we simply mention that the measure can be considered as a
distribution of some functions of certain Markov process.
Furthermore,  we study several properties of QMC and associated
measure. A new notion of $\f$-recurrence of QMC is studied, and it
is established relations between the defined recurrence and the
existing ones.

\end{abstract}

\maketitle

\section{Introduction}

The study of asymptotic behavior of trace-preserving completely
positive maps, also known as quantum channels, is a fundamental
topic in quantum information theory, see for instance
\cite{burgarth,burgarth2,lardi,petulante,petulante2,novotny,novotny2}.
More recently, an important class of quantum channels, namely Open
Quantum Random Walks (OQRWs) has been introduced by S. Attal et al.
\cite{attal} and its long term behavior studied
\cite{attal2,konno,cfrr,pellegrini}. These extensions of Markov
chains, where the process retains some amount of memory which is
encoded by a quantum state.

Let us recall some necessary information about OQRW.
 Let $\mathcal{K}$ denote a separable Hilbert space and let
$\{|i\rangle\}_{i\in \LL}$ be its orthonormal basis indexed by the
vertices of some graph $\LL$ (here the set $\LL$ of vertices might
be finite or countable). Let $\mathcal{H}$ be another Hilbert space,
which will describe the degrees of freedom given at each point of
$\LL$. Then we will consider the space
$\mathcal{H}\otimes\mathcal{K}$. For each pair $i,j$ one associates
a bounded linear operator $B_{j}^i$ on $\mathcal{H}$. This operator
describes the effect of passing from $|j\rangle$ to $|i\rangle$. We
will assume that for each $j$, one has
\begin{equation}\label{Bij}
\sum_i B_{j}^{i*}B_{j}^i=\id,
\end{equation} where, if infinite,
such series is strongly convergent. This constraint means: the sum
of all the effects leaving site $j$ is $\id$. The operators $B^i_j$
act on $\mathcal{H}$ only, we dilate them as operators on
$\mathcal{H}\otimes\mathcal{K}$ by putting
$$
M^i_j=B^i_j\otimes \vert i\rangle\langle j\vert\,.
$$
The operator $M^i_j$ encodes exactly the idea that while passing
from $\vert j\rangle$ to $\vert i\rangle$ on the lattice, the effect
is the operator $B^i_j$ on $\mathcal{H}$.

According to \cite{attal} one has
\begin{equation}\label{Mij}
\sum_{i,j} {M^i_j}^* M^i_j=\id.
\end{equation}

Therefore, the operators $(M^i_j)_{i,j}$ define a completely
positive mapping
\begin{equation}\label{MM}
\cm(\r)=\sum_i\sum_j M^i_j\,\r\, {M^i_j}^*
\end{equation}
on $\ch\otimes\ck$.

In what follows, we consider density matrices on
$\mathcal{H}\otimes\mathcal{K}$ which take the form
\begin{equation}\label{rr}
\rho=\sum_i\rho_i\otimes |i\rangle\langle i|,
\end{equation}
 assuming that $\sum_i\tr(\rho_i)=1$.

For a given initial state of such form, the \textit{Open Quantum
Random Walk (OQRW)} is defined by the mapping $\cm$, which has the
following form
\begin{equation}\label{MM1}
\cm(\rho)=\sum_i\Big(\sum_j B_{j}^i\rho_j B_{j}^{i*}\Big)\otimes
|i\rangle\langle i|.
\end{equation}

By means of the map $\cm$ one defines a family of classical random
process on $\O=\L^\bz_+$. Namely, for any density operator $\r$ on
$\mathcal{H}\otimes\mathcal{K}$ (see \eqref{rr}) the probability
distribution is defined by
$$
\bp_\r(i_0,i_1,\dots,i_n)=\tr(B^{i_n}_{i_{n-1}}\cdots
B^{i_2}_{i_{1}}B^{i_1}_{i_{0}}\rho_{i_0}B^{i_1*}_{i_{0}}B^{i_2*}_{i_{1}}\cdots
B^{i_n*}_{i_{n-1}}).
$$
We point out that this distribution is not a Markov measure
\cite{BBP}.

On the other hand, it is well-known \cite{norris} that to each
classical random walk one can associate certain Markov chain and
some properties of the walk can be explored by the constructed
chain. Therefore, it is natural to construct Quantum Markov chain
associated with OQRW and investigate its properties.

More precisely, the following arises problem: find a quantum Markov
chain (QMC)\footnote{ We note that a Quantum Markov Chain is a
quantum generalization of a Classical Markov Chain where the state
space is a Hilbert space, and the transition probability matrix of a
Markov chain is replaced by a transition amplitude matrix, which
describes the mathematical formalism of the discrete time evolution
of open quantum systems, see
\cite{[AcFr80]}-\cite{accardi2},\cite{fannes2,FM} for more details.}
(or finitely correlated state (FCS)\cite{fannes2}) $\f$ on the
algebra $\ca=\o_{i\in\bz_+}\ca_i$, where $\ca_i$ is isomorphic to
$B(\ch)\o B(\ck)$, $i\in\bz_+$, such that the transition operator
$P$ (see section 3 for details) equal to the mapping
$\cm^*$\footnote{The dual of $\cm$ is defined by the equality
$\tr(\cm(\r)x)=\tr(\r\cm^*(x))$ for all density operators $\rho$ and
observables $x$.} and the restriction of $\f$ to the commutative
subalgebra of $\ca$ coincides with the distribution $\bp_\r$, i.e.
\begin{equation}\label{I1}
\f\big((\id\o|i_0><i_0|)\o\cdots\o(\id\o|i_n><i_n|)\big)=\bp_\r(i_0,i_1,\dots,i_n).
\end{equation}
We note that one can find a state with property \eqref{I1} very
easily, but the question is would that kind of state be difficult to
distinguish as a QMC (or FCS)? Finding such a QMC will allow to
interpret the distribution $\bp_\r$ as a QMC, and to study further
properties of $\bp_p$.

The main aim of this paper is to solve the initial problem and lead
a further investigation in a few of the consequences of the problem.
In what follows, we are going to work within QMC scheme
\cite{Ac,[AcFr80]}, and provide a concrete construction of QMC with
the desired property. We stress that to construct such a state we
define a notion of transpose of QMC (which is impossible to define
with FCS) and using this, one defines QMC associated with $\cm$. The
solution of the problem sheds new light on some properties of the
measure $\bp_\r$. For example, the measure can be considered as a
distribution of some functions of certain Markov process
\cite{nacht,fannes3}. This together with the results of
\cite{fannes3} will allow to compute certain physical quantities
(e.g. entropy) of $\bp_\r$ as well as the further development of the
study of repeated quantum measurements \cite{BJPP} via finitely
correlated states \cite{fannes2}.

 R. Carbone and Y. Pautrat \cite{carbone,carbone2} have recently studied irreducibility and periodicity aspects of the mapping $\cm$
 and as expected, the dynamical behavior of an OQRW is in general
quite different from what is obtained with the usual (closed)
quantum random walk \cite{portugal,salvador}. Simultaneously, an
OQRW is quite different in general from what is obtained with the
usual (closed) quantum random walk \cite{portugal,salvador}. In most
of the existing papers, as a whole, the distribution is not
well-studied therefore in the present paper, we will discover and go
in depth of certain markovianity of the distribution along with
establishing its ergodic properties with $\cm$'s ergodic properties.

Furthermore, the study of the notion of recurrence motivated a large
number of papers extending it in different directions:
see\cite{bourg} for the notion of monitored recurrence for
discrete-time quantum processes; see \cite{werner} for the
recurrence of discrete time unitary evolutions, see \cite{accardi1}
for the recurrence of the quantum Markov chains; see \cite{fagnola}
for the recurrence of the quantum Markovian semigroups. In
\cite{accardi1,accardi2} it was defined a notion of recurrence for
QMC which was based on the transition expectation and initial
projection. When we look at QMC it depends on an initial state and a
transition expectation, therefore, we define the recurrence within
QMC scheme. It turns out that the defined recurrence is connected to
the recurrence of OQRW \cite{BBP,CGL}.  However, the notions of
recurrence elaborated in \cite{BBP,CGL} are purely classical, i.e.
they depend on a classical probability distribution $\bp_\r$ (which
is not necessary to be Markov one, therefore, it has appeared
different phenomena than Markov one) and they are not connected to
the noncommutative observables. In the present paper, we propose to
study $\f^t_\r$-recurrence which could treat more general events in
the non-commutative setting. Namely, one can study
$\f^t_\r$-recurrence of projections rather than
$\id\o|k\rangle\langle k|$. The recurrence of these kinds of
projections can not treated by means of ones investigated in
\cite{BBP,cfrr}. Moreover, the present approach can be also applied
to the case of finitely correlated states (see Remark \ref{FCS1}).


The paper is organized as follows. In section 2, we introduce basic
concepts related to the Open Quantum Random Walks and construct the
corresponding measure which we stress that this measure a'priori is
not a Markov one. As we move along to section 3, we recall a notion
of Quantum Markov Chain (QMC) and its transpose while providing a
construction of QMC. By means of the construction of section 3, we
construct QMCs associated with Open Quantum Random Walks in section
4. It turns out that the transpose of the constructed QMC is
naturally corresponding to the given OQRW. Additionally, in this
section, we solve the posted problem and thoroughly review several
properties of QMC and associated measure. Besides, it is also
provided a construction of finitely correlated states with given
marginal distributions on some commutative algebra. This opens new
perspectives with the results of \cite{BJPP}. In section 5, a new
notion of $\f$-recurrence of QMC is analyzed consequently showing
that the defined recurrence is related to the recurrence
investigated in \cite{BBP,CGL}. Section 6 is devoted to interesting
examples of OQRW for which relations between $\f$-recurrence and the
recurrence in the sense of \cite{CGL} are investigate. Lastly, in
sections 7 and 8 we provide the proofs of the formulated results in
the sections 4 and 5, respectively.

\section{Open Quantum Random Walks}

In this section, we recall basic setup to define the distribution
associated with Open Quantum Random Walks (OQRW).

As before, we consider  the mapping $\cm$ defined by \eqref{MM1}.
Hence, a measurement of the position in $\ck$ would give that each
site $i$ is occupied with probability
\begin{equation*}\label{E:pi}
\sum_j \tr\left(B^i_j \r_j {B^i_j}^*\right)\,.
\end{equation*}

If the measurement is performed after two steps, i.e.
$$
\cm^2(\r)=\sum_{i}\sum_j\sum_k B^i_j B^j_k\, \r_k\,
{B^j_k}^*{B^i_j}^*\otimes \vert i\rangle\langle i\vert\,.
$$
Hence measuring the position, we get the site $\vert i\rangle$ with
probability
\begin{equation*}\label{E:L2}
\sum_j\sum_k \tr\left(B^i_jB^j_k \,\r_k
\,{B^j_k}^*{B^i_j}^*\right)\,.
\end{equation*}

The random walk which is described in this way by the iteration of the
completely positive map $\cm$ is not a classical random walk, it is
a quantum random walk.

The indicated distributions define a measure. Let us construct this
measure.

Let us denote $\O_{\bz_+}=\L^{\bz_+}$, $\O_\bz=\L^{\bz}$, here
$\bz_+$ denotes the set of all non negative integers. A subset of
$\O_{\bz_+}$ (resp. $\O_\bz$) given by
$$
A^{[l,m]}(i_l,i_{l+1},\dots,i_m)=\{\w\in\O_{\bz_+}:\w_l=i_l,\dots,\w_m=i_m\}.
$$
is called \textit{thin cylindrical set}, where $i_k\in\L$,
$k\in\bz_+$. By $\frak{F}$ we denote the $\s$-algebra generated by
thin cylindrical sets.

Since the finite disjoint unions of thin cylinders form an algebra
which generates $\frak{F}$, therefore a measure $\mu$ on $\frak{F}$
is uniquely determined by the values:
$$
\mu_n(A^{[l,n]}(i_l,i_{l+1},\dots,i_n)).
$$
which should satisfy the compatibility conditions, i.e.
\begin{equation}\label{comp1}
\sum_{j\in\L}\mu_{n+1}(A^{[0,n+1]}(i_0,i_{1},\dots,i_n,j))=\mu_n(A^{[0,n]}(i_0,i_{1},\dots,i_n))
\end{equation}
The Kolmogorov's Theorem ensures the existence of the measure $\mu$
on $(\O_{\bz_+},\frak{F})$.

Now for a given $\cm$ (see \eqref{MM1}) and a fixed $\rho$ (see \eqref{rr}), for every $n\in\bn$, we define a measure
$\bp_{\r,n}$ on $\O_n:=\L^{[0,n]}$ as the distribution of the OQRW,
i.e.
\begin{equation}\label{marginal}
\bp_{\r,n}(A^{[0,n]}(i_0,i_{1},\dots,i_n))=\tr(B^{i_n}_{i_{n-1}}\cdots
B^{i_2}_{i_{1}}B^{i_1}_{i_{0}}\rho_{i_0}B^{i_1*}_{i_{0}}B^{i_2*}_{i_{1}}\cdots
B^{i_n*}_{i_{n-1}}).
\end{equation}
The defined measures satisfy the compatibility condition. Indeed,
due to \eqref{Bij} we have
\begin{eqnarray*}
\sum_{j\in\L}\bp_{\r,{n+1}}(A^{[0,n+1]}(i_0,i_{1},\dots,i_n,j))&=&\sum_{j\in\L}\tr(B^{j}_{i_{n}}\cdots
B^{i_2}_{i_{1}}B^{i_1}_{i_{0}}\rho_{i_0}B^{i_1*}_{i_{0}}B^{i_2*}_{i_{1}}\cdots
B^{j*}_{i_{n}})\\[2mm]
&=&\tr\bigg(\bigg(\sum_{j\in\L}\tr(B^{j*}_{i_{n}}B^{j}_{i_{n}})\bigg)B^{i_n}_{i_{n-1}}\cdots
B^{i_2}_{i_{1}}B^{i_1}_{i_{0}}\rho_{i_0}B^{i_1*}_{i_{0}}B^{i_2*}_{i_{1}}\cdots
B^{i_n*}_{i_{n-1}}\bigg)\\[2mm]
&=&\bp_{\r,n}(A^{[0,n]}(i_0,i_{1},\dots,i_n)).
\end{eqnarray*}
Hence, we have the following result.

\begin{prop}\label{MM1} For given OQRW $\cm$ and an initial density
operator $\rho$ there is a unique measure $\bp_\rho$ on
$(\O_{\bz_+},\frak{F})$ with marginal distributions given by
\eqref{marginal}.
\end{prop}

It turns out that the measure $\bp_\rho$ can be extended to
$(\O_{\bz},\frak{F})$ under some conditions. Namely, we have the
following

\begin{prop}\label{MM2} Let $\bp_\rho$ be a measure defined on
$(\O_{\bz_+},\frak{F})$ associated with OQRW $\cm$ and an initial
density operator $\rho$. If $\rho=\sum\limits_i \rho_i\otimes
|i\rangle\langle i|$ is an invariant density operator w.r.t. $\cm$,
then the measure $\bp_\rho$ can be extended to
$(\O_{\bz},\frak{F})$.
\end{prop}

\begin{proof} The invariance of $\rho=\sum\limits_i \rho_i\otimes
|i\rangle\langle i|$ w.r.t. $\cm$ implies
\begin{equation}\label{inv1}
\sum_j B_j^i\rho_j B_j^{i*}=\rho_i,\;\;\;\forall\; i\in\LL.
\end{equation}
Multiply the above equation by $B_i^{i_1}$ and $B_i^{i_1*}$ on the
left and right, respectively, and summing over $i$ we get
\begin{equation}
\sum_{i,j} B_i^{i_1}B_j^i\rho_j B_j^{i*}B_i^{i_1*}=\sum_i
B_i^{i_1}\rho_iB_i^{i_1*}=\rho_{i_1}, \;\forall\; i\in\LL.
\end{equation}
By induction, we can establish that if $\rho$ is invariant, then one
has
\begin{equation}\label{general_stationar}
\sum_{i_1,i_2,\dots, i_n} B_{i_n}^iB_{i_{n-1}}^{i_{n}}\cdots
B_{i_1}^{i_2}\rho_{i_1}B_{i_1}^{i_2*}\cdots
B_{i_{n-1}}^{i_{n}*}B_{i_n}^{i*}=\rho_i,\;\forall\; i\in\LL.
\end{equation}

The measure $\m_\r$ can be extended to $(\O_{\bz},\frak{F})$ if one
has
\begin{equation}\label{comp2}
\sum_{j\in\L}\bp_{\r,{n+1}}(A^{[0,n+1]}(j,i_0,i_{1},\dots,i_n))=\bp_{\r,n}(A^{[0,n]}(i_0,i_{1},\dots,i_n)).
\end{equation}
Let us check the last equality. Indeed, from \eqref{inv1} we have
\begin{eqnarray*}
\sum_{j\in\L}\bp_{\r,{n+1}}(A^{[0,n+1]}(j,i_0,i_{1},\dots,i_n))&=&\sum_{j\in\L}\tr(B^{i_n}_{i_{n-1}}\cdots
B^{i_1}_{i_{0}}B^{i_0}_{j}\rho_{j}B^{i_0*}_{j}B^{i_1*}_{i_{0}}\cdots
B^{j*}_{i_{n}})\\[2mm]
&=&\tr\bigg(B^{i_n}_{i_{n-1}}\cdots
B^{i_2}_{i_{1}}B^{i_1}_{i_{0}}\rho_{i_0}B^{i_1*}_{i_{0}}B^{i_2*}_{i_{1}}\cdots
B^{i_n*}_{i_{n-1}}\bigg)\\[2mm]
&=&\bp_{\r,n}(A^{[0,n]}(i_0,i_{1},\dots,i_n)).
\end{eqnarray*}
This completes the proof.
\end{proof}

\begin{rem}
The existence of invariant density operators for OQRW $\cm$ has been
studied in \cite{cfrr}. One of the sufficient conditions is based on
the irreducibility of the mapping \cite{carbone2}.
\end{rem}

Let us consider a random process $(X_n)$ defined for
$\w=(i_0,i_1,\dots)\in\O_{\bz_+}$ by $X_n(\w)=i_n$. Then the process
$(X_n)$ with distribution $\bp_\r$, in general, is not Markov one
(see \cite[Example 5.1]{BBP}). Moreover, in the existing literature
properties of the measure $\bp_\rho$ are not well-studied.

In this paper, we will show that the measure $\bp_\r$ can be
interpreted as a quantum Markov chain. This allows us to treat such
quantum walks in the framework of QMC.

\section{Quantum Markov Chains}

In this section, we recall the definition of quantum Markov chain
\cite{Ac,[AcFr80],Park}.

For each $ i \in \bz_+ $, (here $\bz_+$ denotes the set of all non
negative integers) let us associate identical copies of a separable
Hilbert space $ \ch $ and $ C^{*} $-subalgebra $ M_{0} $ of $
\mathcal{B}(\ch) $, where $ \mathcal{B}(\ch) $ is the algebra of
bounded operators on $ \ch: $
\[ \ch_{\{i\}} = \ch, \]
\begin{eqnarray}
\ca_{\{i\}} = M_{0} \subset \mathcal{B}(\ch) \textmd{ for each } i
\in \bz_+
\end{eqnarray}

We assume that any minimal projection in $M_0$ is one dimensional.

For any bounded $ \LL \subset \bz_+ $, let

$$ \ca_{\LL} = \underset{i \in \LL}{\bigotimes} \ca_{i}, \ \ \
\ca_{loc}=\bigcup_{\substack{\LL \subset
\bz_+,|\LL|<\infty}}\ca_{\LL}
$$

\begin{eqnarray}
\ca = \overline{\ca_{loc}}=:\bigotimes_{i\in\bz_+}\ca_i
\end{eqnarray}
 where the bar denotes the norm closure.

 For each $ i \in \bz_+ $, let $ J_{i}  $ be the canonical injection of $ M_{0} $ to the $ i $-th component of $ \ca $.
 For each $ \LL \subset \bz_+ $ we identity $ \ca_{\LL} $ as a subalgebra of $
 \ca $.

 The basic ingredients in the construction of a stationary generalized quantum Markov chain in
 the sense of Accardi and Frigerio \cite{[AcFr80]} consist of a \textit{transition expectation} $\ce:M_{0}\o M_{0} \rightarrow M_{0}$
 which is completely positive unital map (i.e.
 $\ce(\id\o\id)=\id$)),
and a state $ \v_{0} $ on $ M_{0} $. In what follows, a pair $
(\v_{0}, \ce) $ is called a \textit{Markov pair}.

A state $\f$ defined on $\ca$ associated with a Markov pair
$(\v_0,\ce)$, is called \textit{Quantum Markov Chain (QMC)} if
\begin{eqnarray}\label{2.5}
\f(x_{0} \o x_{1} \o \dots \o x_{n}) = \v_{0}(\ce( x_{0} \o
\ce(x_{1} \o  \cdots \o \ce(x_{n} \o \id ) \cdots ) ) )).
\end{eqnarray}

Let $\s:M_{0}\o  M_{0}\to M_{0}\o M_{0}$ be the flipping
automorphism defined by $\s(x\o y)=y\o x$. For every transition
expectation $\ce$ one can associate its \textit{transpose} by
$\ce^{t}=\ce\circ\s$. Hence, given a Markov pair  $ (\v_{0}, \ce) $
we naturally associate its \textit{transpose Markov pair}  $
(\v_{0}, \ce^t) $. The QMC corresponding to the pair $ (\v_{0},
\ce^t) $ is called \textit{transpose QMC} of $\f$, and it is denoted
by $\f^t$.\\

To every transition expectation one associates two kinds of Markov
operators (i.e. completely positive, identity preserving map) from
$M_0$ into itself:
 \begin{eqnarray}\label{TO1}
&&P(a)=\ce(\id\o a), \ \ (\textrm{backward transition operator})\\
&&T(a)=\ce(a\o\id), \ \ (\textrm{forward transition operator}).
 \end{eqnarray}

\begin{rem} It is known \cite{[AcFr80]} that in the classical setting $T$
is the identity operator, and $P$ coincides with usual Markov
transition operator.
\end{rem}

\begin{rem} We point out that the quantum Markov chain can be also
treated as a special case of finitely correlated states (FCS) which
were introduced in \cite{fannes2}. Let us  recall the well-known
construction. Let $\ga,\gb$ be two $C^*$-algebras with units
$\id_\ga$,$\id_\gb$, respectively, $\f_0$ be a state on $\gb$, and
$\ce:\ga\o\gb\to\gb$ be a completely positive unital map such that
for all $b\in\gb$ one has
$$
\f_0(\ce(\id_\ga\o b))=\f_0(b).
$$
For each $a\in\ga$ one defines a map $\ce_a:\gb\to\gb$ by setting
$\ce_a(b)=\ce(a\o b)$. The functional
$$
\f(x_1\o \cdots x_n)=\f_0(\ce_{x_1}\cdots\ce_{x_n}(\id_{\gb}))
$$
uniquely defines a state on the $C^*$-algebra
$\bigotimes\limits_{i\in\bn}\ga_i$, where $\ga_i$ is a copy of
$\ga$. The state $\f$ is the \textit{finitely correlated state}
associated to $(\ga,\gb,\ce,\f_0)$. In case, $\ga=\gb$ we will
recover QMC. On the other hand, we stress that, in general, we
cannot define the transpose FCS on the same algebra with the initial
one. Therefore, in what follows, we will work within QMC scheme. \\
\end{rem}

In what follows, by $\ca_{n]}$ we denote the subalgebra of $\ca$,
generated by the first $(n+1)$ factors, i.e.
$$
a_{n]}=a_0\o a_1\o\cdots a_n\o\id_{[n+1}=J_0(a_0)J_1(a_1)\cdots
J_n(a_n),
$$
with $a_0,a_1,\dots,a_n\in M_0$.  It is well known \cite{Ac2} that
for each $n\in\bn$ there exists a unique completely positive
identity preserving mapping $E_{n]}:\ca\to \ca_{n]}$ such that
\begin{eqnarray}\label{CE1}
E_{n]}(a_{m]})=a_0\o\cdots\o
a_{n-1}\o\ce(a_{n}\o\ce(a_{n+1}\o\cdots\o\ce(a_m\o\id)\cdots)), \ \
m>n
\end{eqnarray}

\begin{rem} We notice that if the state $\v_0$ satisfies the following
condition:
\begin{eqnarray} \label{qqq1}
\v_{o}(\ce(\id\o x)) = \v_{0}(x), \ \ x \in M_{0}
\end{eqnarray}
then the Markov pair $(\v_0,\ce)$ defines local states
\begin{eqnarray}\label{2.5-1}
\f_{[i,n]}(x_{i} \o x_{i+1} \o \dots \o x_{n}) = \v_{0}(\ce( x_{i}
\o \ce(x_{i+1} \o  \cdots \o \ce(x_{n} \o \id ) \cdots ) ) )).
\end{eqnarray}
The family of local states $\{\f_{[i,n]}\}$, due to \eqref{qqq1},
satisfies a compatibility condition, and therefore, the state $\f$
is well defined on $ A_\bz:=\bigotimes\limits_{i\in\bz}\ca_i$.
Moreover, $ \f $ is translation invariant, i.e.  it is invariant
with respect to the shift $ \a $, i.e.  $\a (J_{n}(a)) = J_{n+1}(a)
$.
\end{rem}


Recall that by $ \tr $ we denote the trace on $ M_{0} $ which takes
the value 1 at each minimal projection, and let $ \widetilde{\tr} $
be the trace on $ M_{0} \o M_{0}. $ Denote by $
\widetilde{\tr}^{(i)}, \ \ i =1,2, $ the partial traces defined by
\begin{eqnarray}\label{opra_1,2_1}
\widetilde{\tr}^{(1)}(a\o b) = \tr(a)b, \ \
\widetilde{\tr}^{(2)}(a\o b) = \tr(b)a.
\end{eqnarray}

In \cite{Park} it was given a construction of a quantum Markov chain
defined by a set $ \{ K_i\}_{i\in\bn} $ of conditional density
amplitudes \cite{Ac}. Namely, let $ W_{0} \in M_{0} $ be a density
matrix and $ \{K_i\}_{i\in\bn} $ be a set of the Hilbert-Schmidt
operators in $ M_{0} \o M_{0} $  satisfying
\begin{eqnarray}
&&\sum_{i}\|K_i\|^2<\infty, \nonumber\\ \label{KK1} &&
\sum_{i}\widetilde{\tr}^{(2)}(K_iK_i^{*})= \id.
\end{eqnarray}

Then the corresponding transition expectation \cite{[AcFr80]}
\begin{eqnarray}\label{EC1}
\ce(A) = \sum_{i}\widetilde{\tr}^{(2)}(K_iAK_i^{*}), \ \ A\in
M_{0}\o M_{0}.
\end{eqnarray}
and the density operator $W_0$ form a Markov pair $(W_0,\ce)$.

We point out that the transpose transition expectation associated
with \eqref{EC1} has the following form:
\begin{eqnarray}\label{EC12}
\ce^t(A) = \sum_{i}\widetilde{\tr}^{(2)}(K_i\s(A)K_i^{*}), \ \ A\in
M_{0}\o M_{0}.
\end{eqnarray}

Hence,  $(W_0,\ce^t)$ is a Markov pair. We stress that the QMCs
associated with Markov pairs  $(W_0,\ce)$ and  $(W_0,\ce^t)$,
respectively, may have different properties. We will demonstrate
some differences in the next sections.

\begin{rem} We point out if additionally $W_0$ satisfies
\begin{eqnarray}\label{KK2}
\sum_{i}\widetilde{\tr}^{(1)}(K_i^{*}(W_{0}\o \textbf{1})K_i)=
W_{0}.
\end{eqnarray}
Then the associated QMC associated with the pair $(W_0,\ce)$ is well defined on the
algebra $\ca_\bz$.
\end{rem}
%
%
%
%
%

\section{Quantum Markov Chains associated with OQRW}

In this section, we are going to construct QMCs associated with
OQRW.

Let $\cm$ be a OQRW given by \eqref{MM1}. In this section we will
use notations from the previous sections.

Take a density operator $\rho\in B(\ch\o\ck)$, of the form
$$
\rho=\sum_i\rho_i\o|i\rangle\langle i|.
$$
In what follows, we assume that $\r_i\neq 0$ for all $i\in\bn$.

 We are going to construct a QMC associated with $\rho$ and $\cm$.
To do so, we consider the algebra
$$
\ca=\bigotimes_{i\in\bz_+}\ca_i,$$ where $\ca_i=B(\ch\o\ck)$ for all
$i\in\bz_+$.
%

Define the following operators:
\begin{eqnarray}\label{AK1}
&&A_{ij}=\frac{1}{(\tr(\rho_j))^{1/2}}\big(\rho_j^{1/2}\o
|i\rangle\langle j|\big), \ \ \ i,j\in\LL,\\[2mm]
\label{AK2} &&K_{ij}=M^{i*}_{j}\o A_{ij}.
\end{eqnarray}

Let us show that the pair $(\rho,\{K_{ij}\})$ defines a QMC on
$\ca$. Firstly note that the condition \eqref{KK1} follows from
\eqref{Mij}. Indeed, one has
\begin{eqnarray}\label{AK3}
{\tr}^{(2)}\bigg(\sum_{i,j}K_{ij}K_{ij}^*\bigg)&=&\sum_{i,j}M^{i*}_{j}M^i_{j}\frac{\tr(\rho_j\o
|i\rangle\langle i|) }{\tr(\rho_j)}\nonumber\\[2mm]
&=&\id.
\end{eqnarray}

Hence, one can define a quantum Markov chain $\f_\r$ corresponding
to the pair $(\rho,\ce)$, where the transition expectation $\ce$ is
defined by (see \eqref{EC1}):
\begin{eqnarray}\label{AK7}
\ce(x\o y)&=&\sum_{i,j}M^{i*}_{j}x
M^i_{j}\frac{\tr(\rho_j\o|j\rangle\langle j|y)}{\tr(\rho_j)}.
\end{eqnarray}

\begin{rem} We point out that the constructed QMC is not naturally
associated with the given OQRW, since the transition operator
corresponding to $\ce$ \eqref{AK7} is not equal to the dual $\cm^*$
of $\cm$. Indeed, we have
\begin{eqnarray*}
P(x)&=&\sum_{i,j}M^{i*}_{j}
M^i_{j}\frac{\tr(\rho_j\o|j\rangle\langle j|x)}{\tr(\rho_j)}\\[2mm]
&=&\sum_{j}\frac{\tr(\rho_j\o|j\rangle\langle
j|x)}{\tr(\rho_j)}|j\rangle\langle j|
\end{eqnarray*}
which is clearly not equal to
$$
\cm^*(x)=\sum_{i,j}M^{i*}_{j}x M^i_{j}.
$$
\end{rem}

For the sake completeness, let us provide some properties of the QMC
$\f_\r$, which will allow us to distinguish the differences between
the states $\f$ and $\f^t$.


\begin{prop}\label{QMC-ext} The QMC $\f_\r$ associated with the Markov
pair $(\rho,\ce)$ can be extended to $\ca_\bz$. Moreover, $\f_\r$
has the following form:
\begin{eqnarray}\label{AK8}
\f_\r(x_1\o x_2\o\cdots\o
x_n)=\sum_v\tr(\rho_v)\p_v(x_1)\p_v(x_2)\cdots\p_v(x_n),
\end{eqnarray}
where
\begin{eqnarray}\label{AK9}
\p_v(x)=\frac{1}{\tr(\rho_v)}\sum_{i}\tr\big(B^i_v\rho_v(B_v^i)^*\o|i\rangle\langle
i|x\big), \ \ v\in \LL.
\end{eqnarray}
\end{prop}

\begin{rem} From this theorem we infer that the constructed QMC is a
convex combination of product states. On the other hand, this
example is interesting to study the recurrence within QMC.
\end{rem}

Now let us consider the transpose of $\ce$ (see \eqref{AK7}) which is
defined by
\begin{eqnarray}\label{AK72}
\ce^t(x\o y)&=&\sum_{i,j}\frac{\tr(\rho_j\o|j\rangle\langle
j|x)}{\tr(\rho_j)}(M^i_{j})^*y M^i_{j}.
\end{eqnarray}
It is clear that
$$P(x)=\ce^t(\id\o x)=\cm^*(x).$$

We know that the Markov pair $(\r,\ce^t)$ defines the transpose QMC
$\f^t_\r$ on $\ca$.

\begin{thm}\label{inv11}
If $\rho$ is invariant state of $\cm$ (i.e. $\cm(\rho)=\rho$), then
the QMC $\f^t_\r$ can be extended to $\ca_\bz$. Moreover, $\f^t_\r$
is translation invariant.
\end{thm}

\begin{rem}
\begin{itemize}
\item From Proposition \ref{QMC-ext} and Theorem \ref{inv11} we
immediately infer the difference (for example, the extendibility)
between QMC $\f$ and $\f^t_\r$.


\item We notice that taking into account Proposition \ref{MM2} and the last
theorem,  we infer that the measure $\m_\r$ and the state $\f^t_\r$
have the same extendability property.

\end{itemize}
\end{rem}

For any configuration $\w\in\O_{\bz_+}$, we define a product state
$\f_\w$ on $\ca$ as follows:
$$
\f_\w=\bigotimes_{k\in\w}\f_k,
$$
where
\begin{eqnarray}\label{EF2}
\f_k(x)=\frac{\tr(\rho_k\o  |k\rangle\langle k|x)}{\tr(\rho_k)}.
\end{eqnarray}

It is clear that the mapping $\w\to\f_\w$ is measurable on
$(\O_{\bz_+},\frak{F})$.

\begin{thm}\label{QMC2}
The QMC $\f^t_\r$ on $\ca$ has the following form:
\begin{eqnarray}\label{int1}
\f^t_\r=\int_{\O_{\bz_+}}\f_\w d\bp_\r(\w),
\end{eqnarray}
where $\bp_\r$ is the measure given in Proposition \ref{MM1}.
Moreover, one has
\begin{eqnarray}\label{QMC23}
\f^t_\r\big((\id\o|i_0\rangle\langle i_0|)\o\cdots
\o(\id\o|i_n\rangle\langle
i_n|)\big)=\bp_\r(A^{[0,m]}(i_0,\dots,i_m)).
\end{eqnarray}
for any $i_0,\dots,i_n\in\LL$, $n\in\bn$.
\end{thm}

\begin{rem}\label{rho-k}
We note that if one takes $\r=p\o|k\rangle\langle k|$ ($k\in\bn$),
then we can define the corresponding transition expectation as
follows:
\begin{eqnarray}\label{AK721}
\ce^t(x\o y)&=&\sum_{i,j}\tr(\rho\o|j\rangle\langle j|x)(M^i_{j})^*y
M^i_{j}.
\end{eqnarray}
The corresponding, QMC will be denoted by $\f^t_{p,k}$. Moreover, from the
proof of Theorem \ref{QMC2} we infer that
\begin{eqnarray}\label{QMC23}
\f^t_{p,k}\big((\id\o|k\rangle\langle k|)\o(\id\o|i_1\rangle\langle
i_1|)\o\cdots \o(\id\o|i_n\rangle\langle
i_n|)\big)=\bp_{p,k}(k,i_1,\dots,i_n)
\end{eqnarray}
for any $i_1,\dots,i_n\in\LL$, $n\in\bn$. Here
$$
\bp_{p,k}(k,i_1,\dots,i_n)=\tr(B^{i_n}_{i_{n-1}}\cdots
B^{i_2}_{i_{1}}B^{i_1}_{k}pB^{i_1*}_{k}B^{i_2*}_{i_{1}}\cdots
B^{i_n*}_{i_{n-1}}).
$$

\end{rem}

From Theorem \ref{QMC2} we immediately infer that the QMC $\f^t_\r$
is a solution of the posted problem (see Introduction). Moreover,
the theorem yields that the measure $\bp_\r$ can be considered and
treated as a quantum Markov chain while the measure is not Markov
one. This representation shows that the constructed QMC is a
canonical one associated with OQRW, and it opens new perspectives in
the investigation of $\bp_\r$ within QMC scheme.\\

The representation \eqref{int1}  allows us to investigate the
ergodic properties of the state $\f^t_\r$ in terms of the ergodic
properties of the mapping $\cm^*$ and the measure $\bp_\r$, and
vise-versa. We point out that in \cite{carbone2} it has been
investigated several ergodic properties of OQRW $\cm^*$, but there
was not a connection between the erodicities of $\cm^*$ and the
measure $\bp_\r$. Next results shed new light on this question.

Let us first recall some necessary definitions. Consider the measure $\bp_\r$ on $(\O_{\bz_+},\frak{F})$.
By $s:\O_{\bz_+}\to \O_{\bz_+}$ we denote the \textit{shift}
transformation defined by $(s(\w))_n=\w_{n+1}$. Recall that the
measure $\m_\r$ is called:
\begin{itemize}
\item[1.] \textit{ergodic} if for all $A,B\in\frak{F}$ one has
\begin{eqnarray}\label{m-e}
\lim_{n\to\infty}\frac{1}{n}\sum_{k=0}^{n-1}\bp_\r(A\cap
s^{-n}(B))=\bp_\r(A)\bp_\r(B)
\end{eqnarray}

\item[2.] \textit{weak mixing} if for all $A,B\in\frak{F}$ one has
\begin{eqnarray}\label{m-e}
\lim_{n\to\infty}\bp_\r(A\cap s^{-n}(B))=\bp_\r(A)\bp_\r(B)
\end{eqnarray}
\end{itemize}

Now let us recall some necessary notions about the ergodicity of
$C^*$-dynamical systems. A $C^*$-dynamical system  $(\ga,T,\f)$
\footnote{The triple $(\ga,T,\f)$ is called to be a
\textit{$C^*$-dynamical system}, if $\ga$ is a $C^*$-algebra with
unit, $T:\ga\to\ga$ is completely positive unital mapping with an
invariant state $\f$ on $\ga$.} is called
 \begin{itemize}
\item[1.] \textit{ergodic} if for all $x,y\in \ga$ one has
\begin{eqnarray}\label{m-e}
\lim_{n\to\infty}\frac{1}{n}\sum_{k=0}^{n-1}\f(xT^{n}(y))=\f(x)\f(y)
\end{eqnarray}

\item[2.] \textit{weak mixing} if for all $x,y\in \ga$ one has
\begin{eqnarray}\label{m-e}
\lim_{n\to\infty}\f(xT^{n}(y))=\f(x)\f(y).
\end{eqnarray}
\end{itemize}

Several ergodic properties of $C^*$-dynamical systems have been
investigated in \cite{NSZ}.  The Theorem \ref{QMC2} allows us to
establish the following result.

\begin{thm}\label{QMC22} Let $\r$ be an invariant state for $\cm$.
Cosnider the following statements:
\begin{itemize}
\item[(i)] the $C^*$-dynamical system $(B(\ch)\o B(\ck),\cm^*,\r)$ is ergodic (resp. weak
mixing);

\item[(ii)] the $C^*$-dynamical system $(\ca,\a,\f^t_\r)$ is ergodic (resp. weak
mixing);

\item[(iii)] the measure $\bp_\r$ is ergodic (resp. weak mixing).
\end{itemize}
The following implications hold: (i)$\Rightarrow$ (ii)$\Leftrightarrow$(iii).

Moreover, if $\ce^t(B(\ch)\o B(\ck)\o\id)=B(\ch)\o B(\ck)$, then the all
statements are equivalent.
\end{thm}

\begin{rem}

 We notice that in \cite[Sections 3, 4]{carbone2} it was given certain
sufficient conditions for the ergodicity and weak mixing of $\cm^*$.
The last theorem with the results of the mentioned paper opens new
insight to the properties of the measure $\bp_\r$.
\end{rem}

{\bf Observation.} It is known \cite{AFi} that any quantum Markov
state $\f$ admits a representation
$$
\f=\int_{\O_{\bz_+}}\p_\w d\l(\w)
$$
where $\p_\w$ is product states and $\l$ is a Markov measure on
$(\O_{\bz_+},\frak{F})$. Comparing the last one with Theorem
\ref{QMC2}, we point out that in our case the measure $\bp_\r$ is
not necessary to be a Markov one while the state $\f^t_\r$ is a QMC.
Note that the representation \eqref{int1} does not imply that
$\f^t_\r$ is a product state. Since, any state on $B(\mathcal{H})$ is a limit
of convex combination of
vector states.\\

\begin{rem}\label{FCS1} We point out that the construction of the transition
expectation \eqref{AK7} will provide a more general construction of
finitely correlated states associated with successive measuraments
of a finite family $\cj=\{\Phi_a\}_{a\in L}$ of completely positive
maps $\Phi_a:B(\ch)\to B(\ch)$ such that $\Phi=\sum_{a\in L}\Phi_a$
satisfies $\Phi(\id)=\id$. Here the alphabet $L$ describes the
possible outcomes of a single measurement. Let us consider the
measurable space $(\O,\frak{F})$, where $\O=L^{\bz_+}$. A pair
$(\cj,\r)$, where $\r$ is a density matrix on $\ch$, due to
$\Phi(\id)=\id$, defines a distribution on $\O$ by
$$
\bp_{\cj,\r}(a_1,\dots,a_n)=\tr(\r\Phi_{a_1}\circ\cdots\circ\Phi_{a_n}(\id)).
$$
Let $\ck$ denotes a Hilbert space with orthonormal basis
$\{|a\rangle\}_{a\in L}$. By $C(\ch\o\ck)$ we denote a commutative
subalgebra of $B(\ch)\o B(\ck)$ generated by the projections
$\{\id\o|a\rangle\}_{a\in L}$. Then one can see that the
distribution $P_{\cj,\r}$ can be considered as a state on the algebra
$\gc=\bigotimes_{i\in\bz_+}\cc_i$, where $\cc_i$ is a copy of
$C(\ch\o\ck)$. Now using \eqref{AK7} let us define the transition
expectation $\ce:B(\ch)\o B(\ck)\to B(\ch)$ by
\begin{equation}\label{FCS}
\ce(x\o y)=\sum_{a\in L}\tr((\r\o|a\rangle\langle a|)x)\Phi_a(y).
\end{equation}
Assume that $\Phi^*(\r)=\r$, then $(B(\ch), B(\ck),\ce,\r)$ defines
a FCS $\f$ on $\ca$ such that
$$
\f\lceil\gc=\bp_{\cj,\r}, \ \ \ce(\id\o x)=\Phi(x).
$$
Hence, one can study the state $\f$ and $\bp_{\cj,\r}$ all together.
This with results of \cite{BBP,BJPP,fannes3} opens new insight to
the entropy production and recurrence for finitely correlated
states.
\end{rem}

\section{Recurrence of QMC associated with OQRW}

In this section is devoted to the notion of recurrence with QMC
scheme. Furthermore, some examples will be illustrated.

Following \cite{accardi1} we recall a definition of the stopping
time associated with a projection $e\in M_0$, which is a sequence
$\{\tau_k\}$ defined by
\begin{eqnarray*}
&&\t_0=e\o\id_{[1}=J_0(e),\\
&&\t_1=e^{\perp}\o e\o\id_{[2}=J_0(e^{\perp})J_1(e),\\
&&\cdots \cdots,\\
&&\t_k=\underbrace{e^{\perp}\o \cdots \o e^{\perp}}_{k-1}\o
e\o\id_{[k+1}=J_0(e^{\perp})\cdots J_{k-1}(e^{\perp})J_k(e),\\
&&\t^n_\infty=\underbrace{e^{\perp}\o \cdots \o
e^{\perp}}_{n}\o\id_{[n+1}=J_0(e^{\perp})\cdots J_n(e^{\perp}).
\end{eqnarray*}

Since the sequence $\{\t^n_\infty\}$ is decreasing, therefore, its
strong limit exists in $\ca''$ (bicommutant of $\ca$), and it is
denoted by
$$
\t_\infty:=\lim_{n\to\infty}\t^n_\infty
$$

One can see that
\begin{equation}\label{tau1}
\sum_{k\geq 0}\t_k=\id-\t_{\infty},
\end{equation}
where the sum is meant in the strong topology in $\ca''$.

\begin{defn} Let $\f$ be a QMC on $\ca$ associated with the pair $(\v_0,\ce)$. A projection $e$ is called
\begin{enumerate}
\item[(i)] \textit{$\ce$-completely accessible} if
$$
E_{0]}(\t_\infty):=\lim_{n\to\infty}E_{0]}(\t^n_\infty)=0;
$$
\item[(i)]  $\f$-\textit{completely
accessible} if $ \f(\t_\infty)=0$;
\item[(iii)] $\ce$-\textit{recurrent} if $\tr(\ce(e\o\id))<\infty$
and one has
$$
\frac{1}{\tr(\ce(e\o\id))}\tr\left(E_{0]}\bigg(\sum_{n\geq
0}J_0(e)\o\t_n\bigg)\right)=1
$$
\item[(iv)] $\f$-\textit{recurrent} if $\f(J_0(e))\neq 0$ and
$$
\frac{1}{\f(J_0(e))}\f\left(\sum_{n\geq 0}J_0(e)\o\t_n\right)=1
$$
\end{enumerate}
\end{defn}

\begin{defn} Let $\f$ be a QMC on $\ca$ associated with the pair $(\v_0,\ce)$ and $e,f$ be two projections in $M_0$. A projection $f$ is called
\begin{enumerate}
\item[(i)] \textit{$\ce$-accessible} from $e$ if there is $n\in\bn$
such that
$$
E_{0]}(J_0(e)\o\id_{n-1]}\o J_n(f))\neq 0;
$$
\item[(ii)] \textit{$\f$-accessible} from $e$ (we denote it as $e\to^\f f$) if there is $n\in\bn$
such that
$$
\f(J_0(e)\o\id_{n-1]}\o J_n(f))\neq 0.
$$
\end{enumerate}
If $e\to^\f f$ and $f\to^\f e$, then $e$ and $f$ are called
\textit{$\f$-communicate} and one denotes $e\leftrightarrow^\f f$.
\end{defn}

\begin{rem} We notice that the $\ce$-accessibility and
$\ce$-recurrence have been introduced and studied in
\cite{accardi1}. From the definitions one can infer that, due to the
Markov property of $\f$, $\ce$-accessibility and $\ce$-recurrence
imply $\f$-accessibility and $\f$-recurrence, respectively. The
reverse is not true (see Example \ref{Exm1}).
\end{rem}

Now we are going to study several properties of $\f$-accessibility
and $\f$-recurrence, respectively.

\begin{thm}\label{GG1} Let $\f$ be a QMC on $\ca$ associated with the pair $(\v_0,\ce)$.
The following statements hold:
\begin{enumerate}
\item[(i)] $\f(J_n(e))=0$ for all $n\in\bn$ if and only if for every $k\in\bn$ one has $\f(\b^k(\t_\infty))=1$, where
$$\b(a_0\o a_1\o\cdots a_n)=\id\o a_0\o a_1\o\cdots a_n, \ \ \textrm{for any} \ n\in\bn;$$
\item[(ii)] $e$ is $\f$-recurrent if and only if
$\f(J_0(e)\o\t_\infty)=0$. In particular, if $e$ is
$\f$-completely accessible, then $e$ is $\f$-recurrent;

\item[(iii)] if $\f$ is faithful, then $e$ is $\f$-completely
accessible if and only if $e$ is $\f$-recurrent;

\item[(iv)] if all projections in $M_0$ are $\f$-communicating and $e$ is
$\f$-recurrent, then $e$ is $\f$-completely accessible.
\end{enumerate}
\end{thm}

\begin{cor}\label{GG2} Let $\f$ be a QMC on $\ca_\bz$ associated with the pair $(\v_0,\ce)$.
The following statements hold:
\begin{enumerate}
\item[(i)] $\f(e)=0$ if and only if $\f(\t_\infty)=1$;
\item[(ii)] $e$ is $\f$-recurrent if and only if
$e$ is $\f$-completely accessible;

\item[(iii)] if $\f$ is faithful, then $e$ is $\f$-completely
accessible if and only if $\ce$-completely accessible.
\end{enumerate}
\end{cor}

Now we turn our attention of the QMC associated with OQRW. Let $\cm$
be a OQRW given by \eqref{MM1}. In what follows, we will use
notations from the previous sections.

Given a density operator $\r\in B(\mathcal{H}\o\ck)$, one can construct a QMC
$\f^t_\r$ on $\ca$. Due to \eqref{QMC23} the $\f^t_\r$-recurrence of
a projection $\id\o|j\rangle\langle j|$ means the following one:
$$
\sum_{k=1}^\infty\sum_{i_1,\dots,i_{k-1}\atop i_\ell\neq j, 1\leq
\ell\leq k-1}
\bp_\r(A^{[0,k+1]}(j,i_1,\dots,i_{k-1},j))=\bp_\r(A^{[0,0]}(j))
$$
which is equivalent to
$$
\bp_{\r,j}(t_j<\infty)=1
$$
where $t_j(\w)=\inf\{n\in\bn: \ \w_n=j\}$.

Let us take $\r=p\o|k\rangle\langle k|$. Then from Remark
\ref{rho-k} one finds that the $\f^t_{p,k}$-recurrence of
$\id\o|k\rangle\langle k|$  is equivalent to
$\bp_{p,k}(t_k<\infty)=1$. If the projection $\id\o|k\rangle\langle
k|$ is $\f^t_{p,k}$-recurrent for all $p$, then we obtain
LS-recurrence \cite{BBP,cfrr}. Moreover, we obtain that
$\ce$-recurrence implies LS-recurrence. Hence, the
$\f^t_\r$-recurrence is weaker among the recurrences investigated in
\cite{BBP,cfrr}. However, the notions of recurrence studied in
\cite{BBP,CGL} are purely classical, i.e. they depend on a classical
probability distribution (which is not necessary to be Markov one,
therefore, it appeared different phenomena than Markov one) and they are
not connected to the noncommutative observables. In the present
paper, we propose to study $\f^t_\r$-recurrence which could treat
more general events in the non-commutative setting. For example, one
can study $\f^t_\r$-recurrence projections rather than
$\id\o|k\rangle\langle k|$. The recurrence of this kind of
projection can not be studied by means of ones investigated in
\cite{BBP,cfrr}. Moreover, the present approach can be also
applied to the case of finitely correlated states (see Remark
\ref{FCS1}).

\begin{thm}\label{1fi-rec} Let $\cm$ be a OQRW,
$\r=p\o|k\rangle\langle k|$ be an initial density matrix and $Q\in
B(\ch)$ be a projection. Then the following statements hold:
\begin{enumerate}
\item[(i)] if $\tr(pQ)=1$, then the
projection $Q\o|k\rangle\langle k|$ is $\f^t_{p,k}$-recurrent if and
only if $\bp_{p,k}(t_k<\infty)=1$;
\item[(ii)] if $\tr(pQ)<1$, then the
projection $\id-Q\o|k\rangle\langle k|$ is $\f^t_{p,k}$-recurrent.
\end{enumerate}
\end{thm}

From this theorem we infer that the essential difference between the
$\f^t_{p,k}$-recurrence and LS-recurrence for the projection
$\id-Q\o|k\rangle\langle k|$ which is a quantum phenomena that can
not describe by the classical approach, i.e. using the distribution
$\bp_{p,k}$.

\section{Examples}

\begin{exam}
We consider the example given in section 12.1 of \cite{attal}. In
our notation this example is given by $\LL=\{1,2\}$, $\ch=\bc^2$
(with canonical basis $(e_1,e_2)$) and transitions are given by
\[ B^1_{1}=\begin{pmatrix}a & 0 \\ 0 & b\end{pmatrix}\quad
B^{1}_2=\begin{pmatrix}0 & \!\! 1 \\ 0 & 0\end{pmatrix} \quad
B^2_2=\begin{pmatrix}1 & 0 \\ 0 & 0\end{pmatrix}\quad
B^2_1=\begin{pmatrix}c & 0 \\ 0 & d\end{pmatrix}\] where
$|a|^2+|b|^2=|c|^2+|d|^2=1$, $0<|a|^2,|c|^2<1$.

{\bf 1.} Denote
$$
\rho=\begin{pmatrix}1 & 0 \\ 0 & 0\end{pmatrix}\otimes |2\rangle
\langle 2|=:\rho_0\o |2\rangle \langle 2|. $$

A straightforward computation shows that
\begin{eqnarray*}
\cm(\rho)&=&\sum_{i,j=1}^2B^i_j\o |i\rangle \langle j|\big((\rho_0\o
|2\rangle \langle 2|\big) B^{i*}_j\o |j\rangle \langle i|\\[2mm]
&=& \sum_{i=1}^2B^i_2\rho_0B^{i*}_2\o |i\rangle \langle i|.
\end{eqnarray*}

Due to
\begin{equation}\label{B11}
B^1_2\rho_0B^{1*}_2=0, \ \ B^2_2\rho_0B^{2*}_2=\rho_0
\end{equation}
we conclude that $\rho$ is an invariant state for $\cm$.

From \eqref{B11} we immediately infer that
$$
\bp_{\r_0,2}(2,i_1,\dots,i_n)= \left\{ \begin{array}{ll} 1, \ \
\textrm{if} \ \ n=1, i_1=2,\\
0, \ \ \textrm{if} \ \  i_\ell=1, 1\leq \ell\leq n, n\geq 1,
\end{array}
\right.
$$
hence $|2\rangle\langle 2|$ is $\bp_{\r_0,2}$-recurrent.

Now consider the projection $e_Q=Q\o |2\rangle\langle 2|$, where $Q$
is a projection in $B(\ch)$. By denoting $\l=\tr(p_0Q)$, from
\eqref{QMC21},\eqref{EF2} we obtain
\begin{eqnarray*}
\f^t_{\r_0,2}(e_{Q}\o \underbrace{e_{Q}^\perp\o\cdots \o
e_{Q}^\perp}_{n})&=&\sum_{i_1,\dots,i_n}\bp_{\r_0,2}(2,i_1,\dots,i_n)\l\big(1-\l\d_{i_1,2})\cdots \big(1-\l\d_{i_n,2})\nonumber\\[2mm]
&=&\bp_{\r_0,2}(2,\dots,2)\l\big(1-\l)^{n}\\
&=&\l\big(1-\l)^{n}\to 0 \ \ \ \textrm{as} \ \ n\to\infty
\end{eqnarray*}
for any $\l\in(0,1]$. Hence, due to Theorem \ref{GG1} (ii) from the
last relation one gets that the projection $e_Q$ is
$\f^t_{\r_0,2}$-recurrent, if $\tr(\r_0Q)>0$. Note that the
recurrence of the projection $e_Q$ cannot be treated by means of the
classical measure $\bp_{\r_0,2}$.

{\bf 2.} Now consider
$$
\tilde \rho=\begin{pmatrix}0 & 0 \\ 0 & 1\end{pmatrix}\otimes
|2\rangle \langle 2|=:\rho_1\o |2\rangle \langle 2|. $$

Then one can calculate that
\begin{eqnarray}\label{B12}
&& B^1_2\rho_1B^{1*}_2=\r_0,  \ \ B^2_2\rho_1B^{2*}_2=0,
\\[2mm] \label{B22}
&&B^1_2\rho_0B^{1*}_2=|c|^2\r_0, \ \ B^1_1\rho_0B^{1*}_2=|a|^2\r_0.
\end{eqnarray}

Then from \eqref{B12},\eqref{B22} we obtain
\begin{eqnarray*}
\bp_{\r_1,2}(t_2<\infty)&=&\sum_{n=1}^\infty\bp_{\r_1,2}(2,\underbrace{1,\dots,1}_n,2)\\[2mm]
&=&\sum_{n=1}^\infty|c|^2|a|^{2(n-1)}\\[2mm]
&=&\frac{|c|^2}{1-|a|^2}\\[2mm]
&=&1
\end{eqnarray*}
which means that $|2\rangle\langle 2|$ is $\bp_{\r_1,2}$-recurrent.

Let us consider the projection $e_Q=Q\o |2\rangle\langle 2|$. Taking
into account \eqref{B12} one gets
$$
\bp_{\r_1,2}(2,2,i_2,\dots,i_n,2)=0
$$
for any $i_2,\dots,i_n\in\{1,2\}$ ($n\geq 2)$ which with \eqref{B22}
yields
\begin{eqnarray*}
\f^t_{\r_1,2}(e_{Q}\o \underbrace{e_{Q}^\perp\o\cdots \o
e_{Q}^\perp}_{n}\o e_{Q})&=&\sum_{i_1,\dots,i_n}\bp_{\r_1,2}(2,i_1,\dots,i_n,2)\l^2 \prod_{1\leq k\leq n}\big(1-\l\d_{i_k,2})\nonumber\\[2mm]
&=&\sum_{i_2,\dots,i_n}\bp_{\r_1,2}(2,1,i_2\dots,i_n,2)\l^2 \prod_{2\leq k\leq n}\big(1-\l\d_{i_k,2})\nonumber\\[2mm]
&&+\sum_{i_2,\dots,i_n}\bp_{\r_1,2}(2,2,i_2,\dots,i_n,2)\l^2(1-\l)
\prod_{2\leq k\leq n}\big(1-\l\d_{i_k,2})\\[2mm]
&=&\l^2\sum_{k=1}^{n}\bp_{\r_1,2}(2,\underbrace{1,\dots,1}_{n-k+1},\underbrace{2,\dots,2}_k)(1-\l)^{k-1}\\[2mm]
&=&\l^2|c|^2\sum_{k=1}^{n}|a|^{2(n-k)}(1-\l)^{k-1}\\[2mm]
&=&\frac{\l^2|c|^2(|a|^{2n}-(1-\l)^n)}{|a|^2+\l-1}.
\end{eqnarray*}
Now taking into account $\l=\f^t_{\r_1,2}(e_Q)$, from the last
equality with $|a|^2+|c|^2=1$ we find
\begin{eqnarray*}
\frac{1}{\f^t_{\r_1,2}(e_Q)}\sum_{n\geq 0}\f^t_{\r_1,2}(e_{Q}\o
\underbrace{e_{Q}^\perp\o\cdots \o e_{Q}^\perp}_{n}\o
e_{Q})&=&\frac{\l^2|c|^2}{|a|^2+\l-1}\bigg(\frac{|a|^2}{1-|a|^2}-\frac{1-\l}{\l}\bigg)\\[2mm]
&=&\frac{\l^2|c|^2}{\l-|c|^2}\bigg(\frac{|a|^2}{|c|^2}-\frac{1-\l}{\l}\bigg)\\[2mm]
&=&\frac{|a|^2\l}{\l-|c|^2}-\frac{|c|^2(1-\l)}{\l-|c|^2}\\[2mm]
&=&1
\end{eqnarray*}
Hence, for any $Q$ with $\tr(\r_1Q)>0$, the projection $e_Q$ is
$\f^t_{\r_1,2}$-recurrent.

%
%

{\bf 3.} In this case, we assume that $c=0$, and take another
initial state
$$
\rho=\frac{1}{2}\rho_0\o |1\rangle \langle 1|+\frac{1}{2}\rho_0\o
|2\rangle \langle 2|,
$$
here $\r_0$ is given as above.  One can see that
$B^2_1\rho_0B^{2*}_1=0$, $B^1_1\rho_0B^{1*}_1=\rho_0$. Hence, we
conclude that
$$
\tr(B^{i_n}_{i_{n-1}}\cdots B^{i_2}_{i_1}
B^{i_1}_{i_0}\rho_{0}B^{i_1*}_{i_0}B^{i_2*}_{i_1}\cdots
B^{i_n}_{i_{n-1}})=0,
$$
if there is $k_0\in\{0,\dots,n\}$ such that $i_{k_0}\neq i_{k_0-1}$.

Assume that $e^{\perp}=Q\o |1\rangle \langle 1|$. Hence, we have
\begin{eqnarray*}
\f^t_\r(\tau^{n}_\infty)&=&\frac{1}{2}\sum_{i_0,i_1,\dots,i_n}\tr\left(B^{i_n}_{i_{n-1}}\cdots
B^{i_2}_{i_1}
B^{i_1}_{i_0}\rho_{0}B^{i_1*}_{i_0}B^{i_2*}_{i_1}\cdots
B^{i_n}_{i_{n-1}}\right)\f_{i_0}(e^{\perp})\cdots\f_{i_n}(e^{\perp})\\[2mm]
&=&\frac{1}{2}(\tr(\rho_0Q))^n.
\end{eqnarray*}
So, if $\tr(\rho_0Q)<1$ then $e$ is $\f^t$-completely accessible.

Similarly, one gets
$$
\f^t_\r(e\o\tau^{n}_\infty)=\frac{1}{2}(1-\tr(\rho_0Q))(\tr(\rho_0Q))^{n}.
$$
Hence, we infer that if $\tr(\rho_0p)\neq 1$, then $e$ is
$\f^t_\r$-completely accessible iff $\f^t$-recurrent. Otherwise (if
$\tr(\r_0Q)=1$, $e$ is $\f^t_\r$-recurrent, but not
$\f^t_\r$-complete accessible.\\
\end{exam}

\begin{exam}\label{Exm1} In this example, we are going to show that
$ \ce$-recurrence is stronger than $\f$-recurrence. To do so, as an
illustrative example, we are going work with QMC associated given by
\eqref{AK8}.

Let us consider a stationary OQRW on $\bz$ with nearest-neighbor
jumps (see \cite{attal}). Let $\ch$ be a Hilbert space and  $B,C\in
B(\ch)$ such that $B^*B+C^*C=\id$. We define the walk as follows:
assume that $B_{i}^{i-1}=B$ and $B^{i+1}_i=C$ for all $i\in\bz$, all
the others $B^i_j$ being equal to 0. Then one can calculate that
\begin{eqnarray}\label{Ex1}
\cm(\rho)=\sum_j\big(B\o|j-1\rangle\langle j|\rho B^*\o
|j\rangle\langle j-1|+C\o|j+1\rangle\langle j|\rho C^*\o
|j+1\rangle\langle j|\big).
\end{eqnarray}

Take a density operator $\rho\in B(\ch\o\ck)$, of the form
$$
\rho=\sum_i\rho_i\o|i\rangle\langle i|,
$$
with $\rho_i\neq 0$ for all $i$.

Take a projection $e_{Q,k}=Q\o |k\rangle\langle k|$ for some
$k\in\bz$, here $Q\in B(\ch)$ is a projection. Then due to Corollary
\ref{Eee} and \eqref{Ex1} we have
\begin{eqnarray}\label{Ex2}
\ce(e_{Q,k}\o E_{0]}(\tau^n_\infty))&=&\sum_{\ell}B^*\o
|\ell\rangle\langle \ell-1| e_{Q,k} B\o |\ell-1\rangle\langle
\ell|\big(\psi_\ell(e_{Q,k}^{\perp})\big)^{n}\nonumber\\[2mm]
&&+\sum_{\ell}C^*\o |\ell\rangle\langle \ell+1| e_{Q,k} B\o
|\ell+1\rangle\langle \ell|\big(\psi_\ell(e_{Q,k}^{\perp})\big)^{n}\nonumber\\[2mm]
&=&B^*Q B\o |k+1\rangle\langle k+1|\big(\psi_{k+1}(e_{Q,k}^{\perp})\big)^{n}\nonumber\\[2mm]
&&+C^*Q C\o |k-1\rangle\langle
k-1|\big(\psi_{k-1}(e_{Q,k}^{\perp})\big)^{n}.
\end{eqnarray}

Taking into account \eqref{AK9} one finds
\begin{eqnarray}\label{Ex3}
\psi_{\ell}(e^{\perp}_{Q,k})=\left\{
\begin{array}{ll}
1-\frac{\tr(B\rho_{k+1}B^*Q)}{\tr(\rho_{k+1})}, \textrm{if} \ \ \ \ell=k+1\\[4mm]
1-\frac{\tr(C\rho_{k-1}C^*Q)}{\tr(\rho_{k-1})}, \textrm{if} \ \ \
\ell=k-1\\
\end{array}
\right.
\end{eqnarray}
Therefore, from \eqref{Ex2} with \eqref{Ex3} one gets
\begin{eqnarray}\label{Ex4}
\ce(e_{Q,k}\o E_{0]}(\tau^n_\infty))&=&
B^*Q B\o |k+1\rangle\langle k+1|\bigg(1-\frac{\tr(B\rho_{k+1}B^*Q)}{\tr(\rho_{k+1})}\bigg)^{n}\nonumber\\[2mm]
&&+C^*Q C\o |k-1\rangle\langle
k-1|\bigg(1-\frac{\tr(C\rho_{k-1}C^*Q)}{\tr(\rho_{k-1})}\bigg)^{n}.
\end{eqnarray}

Hence, if one has
\begin{eqnarray}\label{Ex5}
0<\tr(B\rho_{k+1}B^*Q)<\tr(\rho_{k+1}), \ \
0<\tr(C\rho_{k-1}C^*Q)<\tr(\rho_{k-1})
\end{eqnarray}
then from \eqref{Ex4} we infer that $\ce(e_{Q,k}\o
E_{0]}(\tau^n_\infty))\to 0$ as $n\to\infty$ which according to
\cite[Theorem 1, (iii)]{accardi1} implies that $e_{Q,k}$ is
$\ce$-recurrent.

Now let us look for the $\f$-recurrence. From \eqref{AK8} we have
$$
\f(e_{Q,k}\o\underbrace{e_{Q,k}^{\perp}\o\cdots \o
e_{Q,k}^{\perp}}_n)=\sum_\ell\tr(\rho_\ell)\p_\ell(e_{Q,k})\big(\psi_{\ell}(e_{Q,k}^{\perp})\big)^n
$$
From
\begin{eqnarray*}
\psi_{\ell}(e_{Q,k})=\left\{
\begin{array}{ll}
\frac{\tr(B\rho_{k+1}B^*Q)}{\tr(\rho_{k+1})}, \textrm{if} \ \ \ \ell=k+1\\[4mm]
\frac{\tr(C\rho_{k-1}C^*Q)}{\tr(\rho_{k-1})}, \textrm{if} \ \ \
\ell=k-1\\
\end{array}
\right.
\end{eqnarray*}
with \eqref{Ex3} we obtain
\begin{eqnarray}\label{Ex6}
\f(e_{Q,k}\o\tau^n_\infty)&=&\tr(B\rho_{k+1}B^*Q)\bigg(1-\frac{\tr(B\rho_{k+1}B^*Q)}{\tr(\rho_{k+1})}\bigg)^{n}\nonumber\\[2mm]
&&+
\tr(C\rho_{k-1}C^*Q)\bigg(1-\frac{\tr(C\rho_{k-1}C^*Q)}{\tr(\rho_{k-1})}\bigg)^{n}.
\end{eqnarray}

Clearly, if \eqref{Ex5} is satisfied then $e_{Q,k}$ is
$\f$-recurrent. This means that under the condition \eqref{Ex5} the
$\ce$-recurrence is equivalent to the $\f$-recurrence.

If one of the following conditions
\begin{enumerate}
\item[(a)] $supp(\rho_{k+1})supp(B)=0$ and
$0<\tr(C\rho_{k-1}C^*Q)<\tr(\rho_{k-1})$,

\item[(b)] $supp(\rho_{k-1})supp(C)=0$ and
$0<\tr(B\rho_{k+1}B^*Q)<\tr(\rho_{k+1})$,

\item[(c)] $supp(\rho_{k+1})supp(B)=0$ and
$supp(\rho_{k-1})supp(C)=0$,
\end{enumerate}
is satisfied, then from \eqref{Ex6} we infer that $e_{Q,k}$ is still
$\f$-recurrent, while it is not $\ce$-recurrent.\\
\end{exam}

\section{Proofs for Section 4}

In this section we collect all the proofs of the formulated Theorems
and Propositions in Section 4.

We first need the following auxiliary fact.

\begin{lem}\label{aux} Let $\rho=\sum\limits_\ell\rho_\ell\o |\ell\rangle\langle
\ell|$. Then one has:\\
\begin{enumerate}
\item[(i)] $\tr(\rho M^{u*}_{v}x
M^u_{v})=\tr(B^u_{v}\rho_v B^{u*}_{v}\o |u\rangle\langle u|x)$;\\

\item[(ii)] $\tr(\rho_v\o |v\rangle\langle v|  M^{i*}_{j}x
M^i_{j})=\tr(B^i_{j}\rho_v B^{i*}_{j}\o |i\rangle\langle
i|x)\d_{vj}$.
\end{enumerate}
\end{lem}

\begin{proof} (i). We have
\begin{eqnarray*}
\tr(\rho M^{u*}_{v}x M^u_{v})&=&\sum_\ell\tr(\rho_\ell \o
|\ell\rangle\langle \ell|M^{u*}_{v}x M^u_{v})\\[2mm]
&=&\sum_\ell\tr\big(B^u_v\o |u\rangle\langle v| (\rho_\ell \o
|\ell\rangle\langle \ell|)(B^{u*}_{v}\o |v\rangle\langle u| x\big)\\[2mm]
&=&\tr(B^u_{v}\rho_v B^{u*}_{v}\o |u\rangle\langle u|x).
\end{eqnarray*}

The equality (ii) can be proven by the same way.
\end{proof}

\begin{proof}[Proof of Proposition \ref{QMC-ext}] The extension of $\f$ on $\ca_\bz$ is defined by
\eqref{2.5-1}. It is compatible, if the condition \eqref{KK2} is
satisfied. Therefore, we check \eqref{KK2} for \eqref{AK2}. Indeed,
we have
\begin{eqnarray}\label{AK4}
{\tr}^{(1)}\bigg(\sum_{i,j}K_{ij}^*(\rho\o\id)K_{ij}\bigg)&=&{\tr}^{(1)}\bigg(\sum_{i,j}M^i_{j}\o A_{ij}^*(\rho\o\id)M^{i*}_{j}\o A_{ij}\bigg)\nonumber\\[2mm]
&=&\sum_{i,j}{\tr}^{(1)}\left[\big(M^i_{j}\rho M^{i*}_{j}\big)\o
A_{ij}^*A_{ij}\right]\nonumber\\[2mm]
&=&\sum_{i,j}{\tr}\big(M^i_{j}\rho
M^{i*}_{j}\big)\frac{\rho_j\o|j\rangle\langle j|}{\tr(\rho_j)}.
\end{eqnarray}

Now using
\begin{eqnarray}\label{AK5}
{\tr}\big(M^i_{j}\rho M^{i*}_{j}\big)&=&{\tr}\big(B^i_{j}\o
|i\rangle\langle j|\rho B^{i*}_{j}|j\rangle\langle
i|\big)\nonumber\\[2mm]
&=&\sum_\ell{\tr}\big(B^i_{j}\o |i\rangle\langle
j|\rho_\ell\o|\ell\rangle\langle \ell| B^{i*}_{j}|j\rangle\langle
i|\big)\nonumber\\[2mm]
&=&{\tr}\big(B^i_{j}\rho_jB^{i*}_{j}|i\rangle\langle
i|\big)\nonumber\\[2mm]
&=&{\tr}\big(B^i_{j}\rho_jB^{i*}_{j}\big)\nonumber\\[2mm]
&=&{\tr}\big(B^{i*}_{j}B^i_{j}\rho_j\big)
\end{eqnarray}
from \eqref{AK5} one finds
\begin{eqnarray}\label{AK6}
{\tr}^{(1)}\bigg(\sum_{i,j}K_{ij}^*(\rho\o\id)K_{ij}\bigg)&=&\sum_{i,j}{\tr}\big(B^{i*}_{j}B^i_{j}\rho_j\big)\frac{\rho_j\o|j\rangle\langle
j|}{\tr(\rho_j)} \nonumber\\[2mm]
&=&\sum_{j}{\tr}\bigg(\sum_iB^{i*}_{j}B^i_{j}\rho_j\bigg)\frac{\rho_j\o|j\rangle\langle
j|}{\tr(\rho_j)}\nonumber\\[2mm]
&=&\sum_{j}{\tr}(\rho_j)\frac{\rho_j\o|j\rangle\langle
j|}{\tr(\rho_j)}\nonumber\\[2mm]
&=&\sum_{j}\rho_j\o|j\rangle\langle
j|\nonumber\\[2mm]
&=&\rho.
\end{eqnarray}
Hence, the above defined QMC can be extended to $\ca_\bz$.

To prove the equality \eqref{AK2}, it is enough to prove for the
case $n=2$ since the general formula can be proved by induction.
From \eqref{AK7} and using Lemma \ref{aux} one finds
\begin{eqnarray*}
\f(x_1\o x_2)&=&\tr(\rho\ce(x_1\o\ce(x_2\o\id)))\\[2mm]
&=&\tr\left(\rho\ce(x_1\o\bigg(\sum_{i,j}M^{i*}_{j}x_2
M^i_{j}\bigg)\right)\\[2mm]
&=&\sum_{i,j}\tr\left(\rho\ce(x_1\o M^{i*}_{j}x_2
M^i_{j})\right)\\[2mm]
&=&\sum_{i,j}\tr\left(\rho\bigg(\sum_{u,v}M^{u*}_{v}x_1
M^u_{v}\bigg)\right)\frac{\tr(\rho_v\o |v\rangle\langle v|
M^{i*}_{j}x_2
M^i_{j})}{\tr(\rho_v)}\\[2mm]
&=&\sum_{u,v}\tr\big(\rho M^{u*}_{v}x_1
M^u_{v}\big)\sum_{i,j}\frac{\tr(\rho_v\o |v\rangle\langle v|
M^{i*}_{j}x_2
M^i_{j})}{\tr(\rho_v)}\\[2mm]
&=&\sum_{u,v}\tr\big(B^u_{v}\rho_v B^{u*}_{v}\o |u\rangle\langle
u|x_1\big)\sum_{i}\frac{\tr\big(B^i_{v}\rho_v B^{i*}_{v}\o
|i\rangle\langle
i|x_2\big)}{\tr(\rho_v)}\\[2mm]
&=&\sum_{v}\tr(\rho_v)\psi_v(x_1)\psi_v(x_2)
\end{eqnarray*}
This completes the proof.
\end{proof}

Using the same idea of the proof we can get the following.

\begin{cor}\label{Eee} For any projection $e\in B(\ch\o\ck^\rho)$ one has
$$
\ce(e\o E_{0]}(\tau^n_\infty))=\sum_{u,v}M^{u*}_{v}e
M^u_{v}\big(\psi_v(e^{\perp})\big)^{n+1}.
$$
\end{cor}

\begin{proof}[Proof of Theorem \ref{inv11}] It is enough to check the equality \eqref{qqq1} for the pair $(\rho,\ce^t)$. From \eqref{AK72} we have
$$
\f_0(\ce^t(\id\o x))=\tr(\rho\cm^*(x))=\tr(\cm(\rho)x)=\tr(\rho
x)=\f_0(x).
$$
Correspondingly, the equality
\begin{eqnarray}\label{inv12}
\f^t_{\r}(x_1\o x_2\o\cdots\o
x_n)=\tr(\rho\ce^t(x_1\o\ce^t(x_2\o\cdots\o\ce^t(x_n\o\id))\cdots))
\end{eqnarray}
defines a QMC on $\ca_\bz$.
\end{proof}

To prove Theorem \ref{QMC2} we need the following auxiliary fact.

\begin{lem}\label{EF0}
One has
\begin{eqnarray}\label{EF1}
\ce^t(x_1\o\ce^t(x_2\o\cdots\o\ce^t(x_n\o\id))\cdots)&=&\sum_{i_1,i_2,\dots,i_n}\big(B^{i_2*}_{i_1}
B^{i_3*}_{i_2}\cdots B^{i_n*}_{i_{n-1}}B^{i_n}_{i_{n-1}}\cdots
B^{i_3}_{i_2}B^{i_2}_{i_1}\o |i_1\rangle\langle
i_1|\big)\nonumber\\[2mm]
&&\times\f_{i_1}(x_1)\cdots\f_{i_n}(x_n),
\end{eqnarray}
where
\begin{eqnarray*}
\f_k(x)=\frac{\tr(\rho_k\o  |k\rangle\langle k|x)}{\tr(\rho_k)}.
\end{eqnarray*}
\end{lem}

\begin{proof} Let us prove for $n=2$. Then from
\eqref{AK72} with \eqref{EF2} we obtain
\begin{eqnarray*}
\ce^t(x_1\o\ce^t(x_2\o\id))&=&\ce^t\left(x_1\o\bigg(\sum_{i,j}M_j^{i*}M^i_j\f_j(x_2)\bigg)\right)\\[2mm]
&=&\sum_{i,j}\ce^t\left(x_1\o(B_j^{i*}B^i_j\o |j\rangle\langle j|)\right)\f_j(x_2)\\[2mm]
&=&\sum_{j}\ce^t\left(x_1\o(\id\o |j\rangle\langle j|)\right)\f_j(x_2)\\[2mm]
&=&\sum_{j}\sum_{u,v}M_v^{u*}(\id\o |j\rangle\langle j|)M_v^{u*}\f_v(x_1)\f_j(x_2)\\[2mm]
&=&\sum_{j,v}(B_v^{j*}B_v^{j*}\o |v\rangle\langle
v|)\f_v(x_1)\f_j(x_2)
\end{eqnarray*}
which shows that \eqref{EF1} is true at $n=2$. General setting can
be proved by the same argument.
\end{proof}

\begin{proof}[Proof of Theorem \ref{QMC2}]  Due to the density argument, it is enough to prove the assertion for local elements of $\ca$.
Namely, one has
\begin{eqnarray}\label{QMC21}
\f^t_\r(x_0\o x_1\o\cdots\o
x_n)&=&\sum_{i_0,i_1,\dots,i_n}\tr\left(B^{i_n}_{i_{n-1}}\cdots
B^{i_2}_{i_1}
B^{i_1}_{i_0}\rho_{i_0}B^{i_1*}_{i_0}B^{i_2*}_{i_1}\cdots
B^{i_n}_{i_{n-1}}\right)\nonumber\\[2mm]&&\times\f_{i_0}(x_0)\cdots\f_{i_n}(x_n).
\end{eqnarray}
The last one immediately follows from Lemma \ref{EF0}.
\end{proof}

\begin{proof}[Proof of Theorem \ref{QMC22}]  (i)$\Rightarrow$ (ii).
Assume that the map $\cm^*$ is weak mixing, i.e. for every state
$\kappa$ one has $\cm^n\kappa\to\r$ weakly.

Due to the density argument, it is enough to prove the statement for
local elements $x,y\in\ca_{loc}$, i.e.
$$
x=x_{i_0}\o\cdots \o x_{i_m}, \ \ y=y_{j_0}\o\cdots\o y_{j_\ell}
$$
Then due to equality \eqref{inv12} one finds
\begin{eqnarray*}
\f^t_\r(x\a^{n}(y))&=&\tr(\r\ce^t(x_1\o\ce^t(x_2\o\cdots
\ce^t(x_m\o(\cm^*)^n(\ce^t(y_1\o\ce^t(y_2\o\cdots
\ce^t(y_\ell\o\id))))\\[2mm]
&\to&\f^t_\r(x)\f^t_\r(y) \ \ \textrm{as} \ \ n\to\infty
\end{eqnarray*}
which yields the assertion. The reverse implication
(ii)$\Rightarrow$(i) follows from the last relation and $\ce^t(B(\ch)\o
B(\ck)\o\id)=B(\ch)\o B(\ck)$.

The implication (ii)$\Rightarrow$ (iii) immediately follows from the
equality \eqref{QMC23}. Therefore, we consider (iii)$\Rightarrow$
(ii).
 Assume that the measure $\m_\r$ is weak mixing. It is enough to prove the statement for local
elements $x,y\in\ca_{loc}$, i.e.
$$
x=x_{i_0}\o\cdots \o x_{i_m}, \ \ y=y_{j_0}\o\cdots\o y_{j_\ell}
$$
Then due to \eqref{QMC21} we have
\begin{eqnarray*}
\f^t_\r(x\a^{n}(y))=\sum_{i_0,\dots,i_n\atop
j_0,\dots,j_\ell}\m_\r(A^{[0,m]}(i_0,\dots,i_m))\cap\s^{-n}(A^{[0,\ell]}(j_0,\dots,j_\ell))
\f_{i_0,\dots,i_m}(x)\f_{j_0,\dots,j_\ell}(y)
\end{eqnarray*}
Now using the weak mixing property of $\m_\r$ one finds
\begin{eqnarray*}
&&\sum_{i_0,\dots,i_n\atop
j_0,\dots,j_\ell}\m_\r(A^{[0,m]}(i_0,\dots,i_m))\cap\s^{-n}(A^{[0,\ell]}(j_0,\dots,j_\ell))
\f_{i_0,\dots,i_m}(x)\f_{j_0,\dots,j_\ell}(y)\\[2mm]
&\to &\sum_{i_0,\dots,i_n\atop
j_0,\dots,j_\ell}\m_\r(A^{[0,m]}(i_0,\dots,i_m)))\m_\r(A^{[0,\ell]}(j_0,\dots,j_\ell))
\f_{i_0,\dots,i_m}(x)\f_{j_0,\dots,j_\ell}(y)\\[2mm]
&=&\bigg(\sum_{i_0,\dots,i_n}\m_\r(A^{[0,m]}(i_0,\dots,i_m))
\f_{i_0,\dots,i_m}(x)\bigg) \bigg(\sum_{j_0,\dots,j_\ell}
\m_\r(A^{[0,\ell]}(j_0,\dots,j_\ell))\f_{j_0,\dots,j_\ell}(y)\bigg)\\[2mm]
&=&\f^t_\r(x)\f^t_\r(y) \ \ \textrm{as} \ \ n\to\infty,
\end{eqnarray*}
 which yields the weak
mixing property of $\f^t_\r$. The ergodicity can be proved by the
same argument.
 This completes the proof.
\end{proof}

\section{Proofs for Section 5}

\begin{proof}[Proof of Theorem \ref{GG1}] (i) Let $\f(J_n(e))=0$ for every $n\in\bn$. For any $k,m\in\bn$ we have
$$\id_{m]}\o\t_k\leq
\id_{m+k-1]}\o J_{m+k}(e)\o\id_{[m+k+1},$$ therefore, one finds
$\f(\beta^m(\t_k))\leq\f(J_{m+k}(e))=0$. Hence, from \eqref{tau1}
one gets $\f(\b^m(\t_\infty))=1$.

Now assume that $\f(\b^m(\t_\infty))=1$ for any $m\in\bn$. Then again
from \eqref{tau1} we obtain
$$
\f\bigg(\b^m\bigg(\sum_{k\geq 0}\t_k\bigg)\bigg)=0,
$$
which implies $\f(\b^m(\t_k))=0$ for all $k\in\bn$. This means
$\f(\b^m(\t_0))=\f(J_m(e))=0$.

(ii) Let $e$ be $\f$-recurrent. Then from the definition and
\eqref{tau1} one finds
$$
\f(J_0(e))=\f\bigg(J_0(e)\o\sum_{k\geq
0}\t_k\bigg)=\f(J_0(e))-\f(J_0(e)\o\t_\infty),
$$
which means $\f(J_0(e)\o\t_\infty)=0$. The reverse implication is
obvious.

(iii) If $\f$ is faithful, then the $\f$-completely accessibility of
$e$ is equivalent to $\t_\infty=0$, then from (ii) we have that $e$
is $\f$-recurrent. Conversely, if $e$ is $\f$-recurrent, then due to
the faithfulness of $\f$ with (ii) one gets $J_0(e)\o\t_\infty=0$,
so $\t_\infty=0$ which means that $e$ is $\f$-completely
accessibility.

(iv) Assume that $e$ is not $\f$-completely accessible, this means
$\f(\t_\infty)>0$. Due to the $\f$-recurrence one has
$\f(J_0(e)\o\t_\infty)=0$, which implies that
$$
\lim_{n\to\infty}\f(J_0(e)\o\t^n_\infty)=0.
$$
The last equality yields that
$$
\lim_{n\to\infty}\f(J_0(e)\o J_1(e)\o\t^n_\infty)=0, \ \ \
\lim_{n\to\infty}\f(J_0(e)\o J_1(e^{\perp})\o\t^n_\infty)=0,
$$
so
$$
\lim_{n\to\infty}\f(J_0(e)\o\id\o\t^n_\infty)=0.
$$
Hence, iterating the last equality, for every $k\in\bn$ we have
\begin{equation}\label{tau2}
\f(J_0(e)\o\id_{k-1]}\o\t_\infty)=0.
\end{equation}
Since $\f(\t_\infty)>0$, then one can find a projection $p\in M_0$
($p\neq 0$) such that $\t_\infty\geq \l p$ for some positive number
$\l$. Then from \eqref{tau2} we infer that
$$
\f(J_0(e)\o\id_{k-1]}\o
p)\leq\frac{1}{\l}\f(J_0(e)\o\id_{k-1]}\o\t_\infty)=0
$$
this implies that $e$ and $p$ are not $\f$-communicate, which is a
contradiction. This completes the proof.
\end{proof}

\begin{proof}[Proof of Corollary \ref{GG2}] Since, $\f$ is a QMC on $\ca_\bz$, then it is a translation invariant state. Therefore, the statement (i) immediately
follows from (i) of Theorem \ref{GG1}. To establish (ii) again due to
the translation invariance of $\f$ we obtain
$$
\f(\t_\infty)=\f(\t_\infty)-\f(e\o\t_\infty)
$$
which by (ii) Theorem \ref{GG1} yields the assertion.

(iii) The faithfulness of $\f$ implies that $e$ is $\f$-completely
accessible iff $\t_\infty=0$, which yields that $e$ is
$\ce$-completely accessible. The reverse implication is obvious.\\
\end{proof}

\begin{proof}[Proof of Theorem \ref{1fi-rec}]
(i) Let us consider the $\f^t_{p,k}$-recurrence of a projection
$e_{Q,k}=Q\o|k\rangle\langle k|$. From \eqref{QMC21},\eqref{EF2}
and Remark \ref{rho-k}, we obtain
\begin{eqnarray}\label{p-rec}
\f^t_{p,k}(e_{Q,k}\o \t_n)&=&\f^t_{p,k}(e_{Q,k}\o
\underbrace{e_{Q,k}^{\perp}\o\cdots \o
e_{Q,k}^{\perp}}_n\o e_{Q,k})\nonumber\\[2mm]
&=&\sum_{i_i,\dots,i_{n}}\bp_{p,k}(k,i_1,\dots,i_n,k)\f_{k}(e_{Q,k})^2\f_{i_1}(e_{Q,k}^\perp)\cdots\f_{i_n}(e_{Q,k}^\perp)\nonumber\\[2mm]
&=&\sum_{i_i,\dots,i_{n}}\bp_{p,k}(k,i_1,\dots,i_n,k)\tr(pQ)^2\big(1-\tr(pQ)\d_{i_i,k}\big)\cdots\big(1-\tr(pQ)\d_{i_n,k}\big)\nonumber\\[2mm]
&=&\sum_{i_i,\dots,i_{n}\atop i_\ell\neq k, 1\leq \ell\leq
n}\bp_{p,k}(k,i_1,\dots,i_n,k)\tr(pQ)^2+\big(1-\tr(pQ)\big)G,
\end{eqnarray}
where $G$ is some expression.

Due to $\tr(pQ)=1$, the last expression \eqref{p-rec} implies that
the projection $e_{Q,k}$ is $\f^t_{p,k}$-recurrent if and only if
$\bp_{p,k}(t_k<\infty)=1$.

(ii) Now assume that $\tr(pQ)<1$, and denote $\l=\tr(pQ)$. Clearly,
$\l\in(0,1)$. Then again from \eqref{QMC21},\eqref{EF2} we find
\begin{eqnarray*}\label{p-rec1}
\f^t_{p,k}(e_{Q,k}^\perp\o \t^n_\infty)&=&\f^t_{p,k}(e_{Q,k}^\perp\o
\underbrace{e_{Q,k}\o\cdots \o
e_{Q,k}}_{n+1})\nonumber\\[2mm]
&=&\sum_{\ell}\bp_{p,k}(\underbrace{\ell,k,\dots,k}_{n+2})\big(1-\l\d_{\ell,k})\l^{n+1}\nonumber\\[2mm]
&=&\bp_{p,k}(\underbrace{k,k,\dots,k}_{n+2})\big(1-\l)\l^{n+1}+\sum_{\ell\neq
k}\bp_{p,k}(\ell,k,\dots,k)\l^{n+1}\\[2mm]
&\leq&\big(1-\l)\l^{n+1}+\sum_{\ell\neq
k}\bp_{p,k}(\ell)\l^{n+1}\\[2mm]
&\leq&\big(1-\l)\l^{n+1}+(1-\bp_{p,k}(k))\l^{n+1}\to 0 \ \ \
\textrm{as} \ \ n\to\infty.
\end{eqnarray*}
Hence, according to Theorem \ref{GG1} (ii) from the last relation
one gets that the projection $\id-P\o|k\rangle\langle k|$ is
$\f^t_{p,k}$-recurrent. This completes the proof.
\end{proof}

\section*{Acknowledgments}
  A. Dhahri acknowledges support by Basic Science Research Program through the
National Research Foundation of Korea (NRF) funded by the Ministry
of Education (grant 2016R1C1B1010008).


\begin{thebibliography}{92}

\bibitem{Ac} L. Accardi, On noncommutative Markov property,
\textit{Funct. Anal. Appl.} {\bf 8} (1975), 1--8.

\bibitem{Ac2} L. Accardi, Local Perturbations of Conditional
Expectations, \textit{J. Math. Anal. Appl.} {\bf 72}(1979), 34--69.

\bibitem{AFi} L. Accardi, F. Fidaleo,
Non homogeneous quantum Markov states and quantum Markov fields,
\textit{J. Funct. Anal.} {\bf 200} (2003), 324--347.

\bibitem{[AcFr80]}
L. Accardi, A. Frigerio, Markovian cocycles, \emph{Proc. Royal Irish
Acad.} {\bf 83A} (1983) 251-263.



\bibitem{accardi1} L. Accardi, D. Koroliuk. Stopping times for quantum Markov chains, \textit{J. Theor. Probab.} {\bf 5}(1992), 521-535.

\bibitem{accardi2} L. Accardi, D. Koroliuk. Quantum Markov chains: The recurrence problem. In book: \textit{Quantum Prob. and Related Topics VII}, 63–73 (1991).

\bibitem{attal} S. Attal, F. Petruccione, C. Sabot, I. Sinayskiy. Open Quantum Random Walks. \textit{J. Stat. Phys.} {\bf 147}(2012), 832-852.

\bibitem{attal2} S. Attal, N. Guillotin-Plantard, C. Sabot. Central Limit Theorems for Open Quantum Random Walks and Quantum Measurement Records.
\textit{Ann. Henri Poincar\'e} {\bf 16} (2015), 15-43.

\bibitem{BBP} I. Bardet, D. Bernard, Y. Pautrat, Passage times, exit times and Dirichlet problems for open quantum walks,
\textit{J. Stat. Phys.} {\bf 167}(2017), 173-204.

\bibitem{BJPP} T. Beboist, V. Jaksic, Y. Pautrat, C.A. Pillet, On entropy production of repeated quantum measurament I. Genera Theory,
\textit{Commun. Math. Phys.}, DOI: 10.1007/s00220-017-2947-1.

\bibitem{bourg} J. Bourgain, F. A. Gr\"unbaum, L. Vel\'azquez, J. Wilkening. Quantum recurrence of a subspace and operator-valued Schur functions,
\textit{Commun. Math. Phys.}, {\bf 329} (2014) 1031-1067.

\bibitem{burgarth} D. Burgarth, G. Chiribella, V. Giovannetti, P. Perinotti, K. Yuasa, Ergodic and Mixing Quantum Channels in Finite Dimensions.
\textit{New J. Phys.} {\bf 15}, 073045 (2013).

\bibitem{burgarth2} D. Burgarth, V. Giovannetti, The generalized Lyapunov theorem and its application to quantum channels.\textit{ New Journal of Physics} 9 (2007) 150.

\bibitem{carbone} R. Carbone, Y. Pautrat. Homogeneous open quantum random walks on a lattice. \textit{J. Stat.
Phys.} {\bf 160}(2015), 1125-1152.

\bibitem{carbone2} R. Carbone, Y. Pautrat. Open quantum random walks: reducibility, period, ergodic properties. \textit{Ann. Henri
Poincar\'e} {\bf 17}(2016), 99-135.

\bibitem{CGL} S.L. Carvalho, L.F. Guidi, C. F. Lardizabal. Site recurrence of open and unitary quantum walks on the line,
\textit{Quantum Inf. Process.} {\bf 16 }(2017), 17.

\bibitem{fagnola} F. Fagnola, R. Rebolledo. Transience and recurrence of quantum Markov semigroups, \textit{Probab. Theory Relat. Fields}  {\bf 126} (2003), 289--306.


\bibitem{fannes2}
M. Fannes, B. Nachtergaele, R.F. Werner, Finitely correlated states
on quantum spin chains, \textit{Commun. Math. Phys.} {\bf 144}
(1992), 443--490.

\bibitem{fannes3}
M. Fannes, B. Nachtergaele, L. Slegers, Functions of Markov
processes and algebaraic measures, \textit{Rev. Math. Phys.} {\bf 4}
(1992), 39--64.


\bibitem{FM} F. Fidaleo, F. Mukhamedov, Diagonalizability of non homogeneous quantum Markov
states and associated von Neumann algebras, \textit{Probab. Math.
Stat.} {\bf 24} (2004), 401--418.


\bibitem{werner} F. A. Gr\"unbaum, L. Vel\'azquez, A. H. Werner, R. F. Werner. Recurrence for discrete time unitary evolutions. \textit{Commun. Math. Phys.} {\bf 320}(2013), 543–569.


\bibitem{konno} N. Konno, H. J. Yoo. Limit theorems for open quantum random walks. \textit{J. Stat. Phys.} {\bf 150} (2013), 299-319.


\bibitem{cfrr} C. F. Lardizabal, R. R. Souza. On a class of quantum channels, open random walks and recurrence. \textit{J. Stat. Phys.} {\bf 159}(2015), 772-796.

\bibitem{lardi} C. F. Lardizabal. A quantization procedure based on completely positive maps and Markov operators. \textit{Quantum Inf. Process.} {\bf 12 }(2013), 1033-1051.

\bibitem{petulante} C. Liu, N. Petulante. On Limiting distributions of quantum Markov chains. \textit{Int. J. Math. and Math. Sciences.} {\bf 2011}(2011), ID 740816.

\bibitem{petulante2} C. Liu, N. Petulante. Asymptotic evolution of quantum walks on the N-cycle subject to decoherence on
both the coin and position degrees of freedom. \textit{Phys. Rev. E}
{\bf 81}(2010), 031113.

\bibitem{nacht} B. Nachtergaele, Working with quantum Markov states and their classical analogoues, in \textit{Quantum Probability and Applications
V}, Lect. Notes Math. {\bf 1442}, Springer-Verlag, 1990, pp.
267--285.


\bibitem{NSZ} C. P. Niculescu, A. Str\"{o}h, L. Zsid\'{o}, Noncommutative
extensions of classical and multiple recurrence theorems, \textit{J.
Operator Theory} {\bf 50} (2003), 3--52.

\bibitem{nielsen} M. A. Nielsen, I. L. Chuang. \textit{Quantum computation and quantum information}. Cambridge University Press, 2000.
\bibitem{norris} J. R. Norris. \textit{Markov chains}. Cambridge University Press, 1997.
\bibitem{novotny} J. Novotn\'y, G. Alber, I. Jex. Asymptotic evolution of random unitary operations. \textit{Cent. Eur. J. Phys.} {\bf 8}(2010), 1001-1014.
\bibitem{novotny2} J. Novotn\'y, G. Alber, I. Jex. Asymptotic properties of quantum Markov chains. \textit{J. Phys. A: Math. Theor.} {\bf 45}(2012) 485301.

\bibitem{Park} Y.M. Park, H.H. Shin, Dynamical entropy of generalized quantum Markov chains
over infinite dimensional algebras, \textit{J. Math. Phys.} {\bf 38}
(1997), 6287--6303.
\bibitem{pellegrini} C. Pellegrini. Continuous time open quantum random walks and non-Markovian Lindblad master equations. \textit{J. Stat. Phys.} {\bf 154}(2014), 838-865.

\bibitem{portugal} R. Portugal. \textit{Quantum walks and search algorithms}. Springer, 2013.

\bibitem{salvador} S. E. Venegas-Andraca. Quantum walks: a comprehensive review. \textit{Quantum Inf. Process.} {\bf 11}(2012), 1015-1106.

\end{thebibliography}
\end{document}